\documentclass[11pt,a4paper]{article}
\usepackage[T1]{fontenc}
\usepackage{graphicx}
\usepackage{mathtools}
\usepackage{amssymb}
\usepackage{amsthm}
\usepackage{thmtools}
\usepackage{microtype}
\usepackage{ragged2e}
\justifying 
\usepackage{xcolor}
\usepackage{nameref}
\usepackage{hyperref}
\hypersetup{
	colorlinks=true,
	linkcolor=blue,
	citecolor=blue,
	filecolor=magenta,      
	urlcolor=cyan,
}
\usepackage{algorithm}
\usepackage{algorithmicx}
\usepackage{algpseudocode} 
\usepackage{setspace}
\usepackage{amsfonts}
\title{Augmented Lagrange method for optimal control problems of parabolic equation with state constraints}
\author{Weilong You \thanks{College of Science, University of Shanghai for Science and Technology,
		Shanghai 200093, China (email: {\tt 1127880510@qq.com}).}
	\and Fu~Zhang\thanks{Corresponding author. College of Science, University of Shanghai for Science and Technology,
		Shanghai 200093, China (email: {\tt fuzhang82@gmail.com}).
		F.~Zhang was partially supported 
		by the National Natural Science Foundation of China (No.  12071292).}}
\date{\today}
\allowdisplaybreaks

\newtheorem{theorem}{Theorem}[section] 
\newtheorem{lemma}[theorem]{Lemma}     
\newtheorem{definition}{Definition}
\newtheorem{assumption}{Assumption}
\begin{document}
	\maketitle
	\textbf{Abstract} The augmented Lagrange method is employed to address the optimal control problem involving pointwise state constraints in parabolic equations. The strong convergence of the primal variables and the weak convergence of the dual variables are rigorously established. The sub-problems arising in the algorithm are solved using the Method of Successive Approximations (MSA), derived from Pontryagin's principle. Numerical experiments are provided to validate the convergence of the proposed algorithm.
		\\
		\\
		\\
	\section{Introduction}
    In this paper, the pointwise control problem of parabolic equation with state constraints of\\
    \begin{equation}
    	\begin{aligned}
    \min J(y,u,v):=\frac{1}{2}\parallel y(\cdot,T)-y_d\parallel_{L^2(\Omega)}^2+\frac{\alpha}{2}\parallel u\parallel_{L^2(\Omega_T)}^2+\frac{\beta}{2}\parallel v\parallel_{L^2(\Sigma_T)}^2,
        \end{aligned}
     \label{P}
     \tag{P}
     \end{equation}
    subject to\\
    \begin{align*}
    y_t+Ay&=u \;\; \rm{in} \; \Omega_T,\\
    \partial_{\nu_A}y&=v \;\;  \rm{on} \; \Sigma_T,\\
    y(\cdot,0)&=y_0 \;  \rm{on} \; \overline{\Omega}\\
    y&\leq \psi \;\;  \rm{in} \; \Omega_T,\\
      u_a\leq u&\leq u_b \;\;   \rm{a.e. \; in} \; \Omega_T,\\
    v_a\leq v&\leq v_b \;\;  \rm{a.e. \; on} \; \Sigma_T.\\
        \end{align*}
    Where $A$ is a second-order elliptic operator, the setting of the problem will be described in Sub-section 2.1.\\
    In above optimal control problem, the Lagrange multiplier associated with the inequality constraint on the state belongs to Radon measure space \( \mathcal{M}(\overline{\Omega}_T) \) the dual space of \( C(\overline{\Omega}_T) \). As a result, the regularity of the Lagrange multiplier is relatively low, which can pose challenges when solving the optimal control problem using numerical methods. Approaches to address this issue are presented in several pieces of literature. The first is Moreau-Yosida regularization, proposed in \cite{Ito2000A}, which is commonly applied to optimal control problems with elliptic state constraints, as discussed in greater detail in \cite{Bergounioux2000A, Hinterm2006P, Ito2003S, Kunisch2002P}. The second approach is Lavrentiev regularization, which is primarily used for elliptic state constraints \cite{Meyer2007O,Meyer2006O,Tr2004R,Tr2009A} but has also been mentioned in the context of optimal control problems with parabolic state constraints \cite{Neitzel2009O,Neitzel2012N}. Additionally, interior-point methods have also been employed to solve such problems. These methods are detailed in \cite{Hinze2011D, Kruse2015A, Pearson2017F, Schiela2009B, Schiela2013A, Schiela2011A} for elliptic state constraints and in \cite{Benedix2011A,Pearson2017F,Pr2007A} for parabolic state constraints.
    
    Despite their utility, these methods have certain limitations. Moreau-Yosida regularization introduces a smoothing term to address the nonsmoothness of the state constraints, but it increases problem complexity and can result in high computational costs, especially for high-dimensional problems. The choice of the regularization parameter also heavily influences solution quality and convergence speed. Similarly, Lavrentiev regularization modifies the constraints to be smoother, but this can lead to reduced solution accuracy and thus requires careful tuning of the regularization parameter. Furthermore, the method may be cumbersome to implement and is not suitable for all types of state constraints. Interior-point methods, while effective for smooth constraints, may struggle with low-regularity Lagrange multipliers, especially in cases involving discontinuous or nonsmooth constraints, which can affect their convergence and performance.
    
    In this paper, we propose the use of the \textbf{augmented Lagrange method} to solve the problem \( \eqref{P} \). The augmented Lagrange multiplier method is well-established in finite-dimensional optimization problems. In \cite{B2019L}, the author introduced an augmented Lagrange method for constrained optimization problems in Banach spaces, presenting an example of a simple elliptic state-constrained optimal control problem. In \cite{Karl2018A}, the author provided a detailed analysis of the augmented Lagrange method for elliptic state-constrained optimal control problems, including a specific convergence analysis. We adapt the method for parabolic state constraints and analyze its convergence. To solve the augmented Lagrange sub-problem, we use the \textbf{method of successive approximate (MSA)} \cite{Li2017M,Chernous1982M}, based on Pontryagin's principle, with the Hamiltonian function detailed in \cite{Casas1997P,Raymond1999H}.
    The structure of this paper is as follows: In Section 2, we present the preliminary results related to the original problem \( \eqref{P} \). In Section 3, we discuss the application of the augmented Lagrange method to this problem. Subsection 3.1 provides a brief introduction to the algorithm. In Subsection 3.2, we present the multiplier update rule, noting that the multiplier is updated only when the index decreases, and the required index is given in Subsection 3.3. In Section 3.4, we introduce the augmented Lagrange algorithm in detail and prove that it terminates without running indefinitely, using the index defined earlier as the stopping condition. Subsection 3.5 provides the convergence proof of the algorithm. Section 4 presents the MSA algorithm for solving the augmented Lagrange sub-problem. Finally, Section 5 presents a numerical example to illustrate the method.\\
    \textbf{Notation} Throughout this paper, \( (\cdot, \cdot)_{\Omega} \) denotes the inner product in \( L^2(\Omega) \), \( (\cdot, \cdot)_{\Sigma} \) denotes the inner product in \( L^2(\Sigma) \), while \( \langle \cdot, \cdot \rangle \) represents the duality pairing between \( \mathcal{M}(\overline{\Omega}_T) \) and \( C(\overline{\Omega}_T) \). Here, \( \mathcal{M}(\overline{\Omega}_T) \) refers to the dual space of \( C(\overline{\Omega}_T) \), which corresponds to the space of Radon measures on \( \overline{\Omega}_T \).
	\section{Preliminary results}
	\subsection{Setting of the control problem}
	Let $\Omega\subset \mathbb{R}^N$, $(N\in\{2,3\})$ be open and bounded with $C^{1,1}-$boundary $\Gamma$. Given $0<T<+\infty$, we set $\Omega_T:=\Omega \times (0,T)$ and $\Sigma_T:=\Gamma \times (0,T)$. Let $\mathcal{Y}$ denote the space $\mathcal{Y}:=\mathcal{W}(0,T;L^2(\Omega),H^1(\Omega)) \cap C(\overline{\Omega}_T)$, where $\mathcal{W}(0,T;L^2(\Omega),H^1(\Omega))$ denote the Hilbert space $\{y : y \in L^2(0,T;H^1(\Omega)) ,y_t \in L^2(0,T;H^{-1}(\Omega))\}$ equipped with the norm 
	$$\parallel y\parallel_{W(0,T;L^2(\Omega),H^1(\Omega))}:=\sqrt{\parallel y\parallel_{L^2(0,T;H^1(\Omega))}^2+\parallel y_t\parallel_{L^2(0,T;H^{-1}(\Omega))}^2}$$
	Moreover we set $\mathcal{U}:=L^r(\Omega_T),\;\mathcal{V}:=L^s(\Sigma_T)$ with $r > N/2+1, \; s > N+1$. We want to solve the following state-constrained optimal control problem: Minimize\\
	$$J(y,u,v):=\frac{1}{2}\parallel y(\cdot,T)-y_d\parallel_{L^2(\Omega)}^2+\frac{\alpha}{2}\parallel u\parallel_{L^2(\Omega_T)}^2+\frac{\beta}{2}\parallel v\parallel_{L^2(\Sigma_T)}^2$$\\
	over all $(y,u,v)\in \mathcal{Y} \times \mathcal{U} \times \mathcal{V}$ subject to the parabolic equation\\
	\begin{equation}
	\begin{aligned}
	(y_t+Ay)(x,t)&=u(x,t) \; \rm{in} \; \Omega_T,\\
	(\partial_{\nu_A}y)(x,t)&=v(x,t) \; \rm{on} \; \Sigma_T,\\
	y(\cdot,0)&=y_0 \;\;\;\;\;\;\;\; \rm{on} \; \overline{\Omega},\
\end{aligned}
\label{B}
\end{equation}
	and subject to the pointwise state and control constraints\\
	\begin{equation}
		\begin{aligned}
	y(x,t)&\leq \psi(x,t) \; \rm{in} \; \Omega_T,\\
	u_a(x,t)\leq u(x,t)&\leq u_b(x,t) \;\rm{a.e. \; in} \; \Omega_T,\\
	v_a(x,t)\leq v(x,t)&\leq v_b(x,t) \; \rm{a.e. \; on} \; \Sigma_T.\\
\end{aligned}
\label{A}
	\end{equation}
We say that $(\bar{y},\bar{u},\bar{v})$ is a solution to \eqref{P}, if the triplet \( (\bar{y},\bar{u},\bar{v}) \) satisfies the inequality constraints \eqref{A}, where \( \bar{y} \) is the weak solution to the associated parabolic equation \eqref{B}, and the functional \( J(\bar{y},\bar{u},\bar{v}) \) achieves the minimum, i.e.,  
\[
J(\bar{y}, \bar{u}, \bar{v}) \leq J(y, u, v),
\]
for any $(y, u, v) \in \mathcal{Y} \times \mathcal{U} \times \mathcal{V}, \; 
\text{where } (y, u, v) \text{ satisfies } \eqref{A}, \text{ and } y \text{ satisfies } \eqref{B}.$

 The precise definition of a weak solution to the parabolic equation is provided below.
 	\begin{definition}
 	A function $y\in \mathcal{W}(0,T;L^2(\Omega),H^1(\Omega))$ is called a weak solution of the state equation if it holds\\
 	\[
 	\begin{aligned}
 		&\int_0^T\int_\Omega y_t(x,t)w(x,t)dxdt 
 		- \int_0^T\int_\Sigma v(x,t)w(x,t)dSdt \\
 		&\quad + \int_0^T\int_\Omega \sum_{i,j=1}^{N} a_{ij}(x)\partial_{x_i}y(x,t)\partial_{x_j}w(x,t)dxdt \\
 		&= \int_0^T\int_\Omega u(x,t)w(x,t)dxdt,
 		\quad \forall w \in L^2(0,T;H^1(\bar{\Omega})).
 	\end{aligned}
 	\]
 \end{definition}
	In the sequel, we will work with the following set of standing assumptions.
	\begin{assumption}
    1) The given data satisfy $y_d \in L^2(\Omega)$, $\alpha >0$, $\beta>0$, $u_a,u_b \in L^r(\Omega_T)$, $v_a,v_b \in L^s(\Sigma_T)$ with $u_a<u_b$, $\parallel u_a \parallel_{L^r(\Omega_T)}<\infty$, $\parallel u_b \parallel_{L^r(\Omega_T)}<\infty$, $v_a<v_b$, $\parallel v_a \parallel_{L^s(\Sigma_T)}<\infty$, $\parallel v_b \parallel_{L^r(\Sigma_T)}<\infty$, $\psi \in  C(\overline{\Omega}_T)$.\\
	2) The differential operator A is given by \\
	$$Ay(x,t)=-\sum_{i,j=1}^{N}\partial_{x_j}(a_{ij}(x)\partial_{x_i}y(x,t)),$$
	with $a_{ij}(x) \in L^\infty(\Omega)$, $a_{ij}(x)=a_{ji}(x)$. The operator $A$ is assumed to be uniformly elliptic, i.e., there is a positive constant $\theta$ such that \\
	$$\sum_{i,j=1}^{N}a_{ij}(x)\xi _i\xi _i \geq \theta\parallel\xi\parallel^2$$ for almost all $x \in \Omega$ and all $\xi \in R^N$.\\
	3) The co-normal derivative $\partial_{\nu_A}y$ is given by\\
	$$\partial_{\nu_A}y=\sum_{i,j=1}^{N}a_{ij}(x)\partial_{x_i}y(x,t)\nu_j(x),$$
	where $\nu_A=(\nu_1,\dots,\nu_N)$ denotes the outward unit normal vector on $\Gamma$.
		\end{assumption}
	The existence lemma for the weak solution of the parabolic equation follows from  Assumption 1.
	\begin{lemma}
	\textbf{(Existence of weak solution)}When the above assumptions fulfilled, for every $u \in L^r(\Omega_T),\; v \in L^s(\Sigma_T)$, there exists a weak solution $y \in W(0,T;L^2(\Omega),H^1(\Omega))$ of the state equation.\cite{Evans2010P}
\end{lemma}
\begin{theorem}
   Let $\Omega \subset \mathbb{R}^n$ be open and bounded with $C^{1,1}$-boundary and let Assumption 1 be satisfied. Then, for every $f \in L^r(\Omega_T)$, $g \in L^s(\Sigma_T)$ and $y_0 \in C(\overline{\Omega})$ with $r > n/2+1, \; s > n+1$, the parabolic partial differential equation\\   
\begin{equation}
	\begin{aligned}
		y_t + Ay &= f, & \rm{in} \; \Omega_T, \\
		\partial_{\nu_A} y &= g, & \rm{on} \; \Sigma_T, \\
		y(\cdot, 0) &= y_0, & \rm{on} \;\;\; \overline{\Omega},
	\end{aligned}
	\label{1}
\end{equation}
     admits a unique weak solution $y \in W(0,T; L^2(\Omega), H^1(\Omega)) \cap C(\overline{\Omega}_T)$, and there exists a constant $C_\infty > 0$ such that
 \begin{equation}
    \parallel y \parallel_{W(0,T; L^2, H^1)} + \parallel y \parallel_{C(\overline{\Omega}_T)} \leq C_\infty \left( \parallel f \parallel_{L^r(\Omega_T)} +\parallel g \parallel_{L^s(\Sigma_T)} +\parallel y_0 \parallel_{C(\overline{\Omega})} \right),
   	\label{2}
 \end{equation}
    where $C_\infty$ does not depend on $f$, $g$, or $y_0$.\\
    \end{theorem}
    \begin{proof}
    This proof is due to Theorem 1.40 on page 49 of \cite{Hinze2009O}.
       \end{proof}
	From this theorem we know that, if $f_n \to f$ in $L^r(\Omega_T)$ and $g_n \to g$ in $L^s(\Sigma_T)$, then the solution $y_n$ of  \eqref {1} with data  $(f_n,g_n)$ converges strongly in ${\mathcal{W}(0,T;L^2(\Omega),H^1(\Omega))}$ and $C(\overline{\Omega}_T)$ to the solution $y$ of \eqref {1} to the data $(f,g)$. It is not hard to prove that the control-to-state mapping $\mathcal{S}:(u,v) \mapsto y$ from $L^r(\Omega_T) \times L^s(\Sigma_T)$ to $\mathcal{W}(0,T;L^2(\Omega),H^1(\Omega)) \cap C(\overline{\Omega}_T)$ is continuous.\\
	Now we represent state and control constraints in terms of feasible sets,
	\[
	\begin{aligned}
	\mathcal{U}_{ad}&=\{u \in L^r(\Omega_T) | u_a(x,t)\leq u(x,t)\leq u_b(x,t) \; \rm{a.e. \; in} \; \Omega_T\},\\
	\mathcal{V}_{ad}&=\{v \in L^s(\Sigma_T) | v_a(x,t)\leq v(x,t)\leq v_b(x,t) \; \rm{a.e. \; on} \; \Sigma_T\},\\
	\mathcal{Y}_{ad}&=\{y \in C(\overline{\Omega}_T) | y(x,t)\leq \psi(x,t) \; \rm{in} \; \Omega_T \}.
    \end{aligned}
     \]
	The feasible set of the optimal control problem is denoted by
	$$\mathcal{F}_{ad}=\{(y,u,v) \in \mathcal{Y}\times \mathcal{U}\times \mathcal{V} | (y,u,v)\in \mathcal{Y}_{ad}\times \mathcal{U}_{ad}\times \mathcal{V}_{ad} ,\; y=\mathcal{S}(u,v) \}.$$
	We usually think of $\mathcal{U}_{ad}\subset \mathcal{U}$, $\mathcal{V}_{ad}\subset \mathcal{V}$, $\mathcal{Y}_{ad}\subset \mathcal{Y}$, as non-empty closed convex sets, which makes the following analysis easier.
	\subsection{Existence of solutions}
	Under the above assumptions, the existence and uniqueness of the constrained control problem can be obtained by the following theorem.
	\begin{theorem}
	 Let Assumption 1 hold. Assume that the feasible set $\mathcal{F}_{ad}$ is nonempty. Then problem \eqref {P} has a uniquely optimal solution $(\bar{y},\bar{u},\bar{v})$.
\end{theorem}
	\begin{proof}
	This result is given in Theorem 1.43 of \cite{Ito2000A}.
\end{proof}
	\subsection{Optimality conditions}
	In order to guarantee the existence of Lagrange multipliers for state-constrained optimal control problems, we will assume the following Slater conditions hold.
	\begin{assumption}
		We assume that there exists $\hat{u}\in \mathcal{U}_{ad}, \; \hat{v}\in \mathcal{V}_{ad}$ and $\sigma>0$ such that for $\hat{y}=\mathcal{S}(\hat{u},\hat{v})$ it holds 
	$$\hat{y}(x,t)\leq \psi(x,t)-\sigma \;\;\;\forall(x,t)\in \Omega_T.$$
		\end{assumption}
	Under the basis of Assumption 2, we give the first-order necessity condition for the problem \eqref{P} in the following theorem.
		\begin{theorem}
	 Let $(\bar{y},\bar{u},\bar{v})$ be a solution of the problem \eqref{P}, let Assumption 2 be fulfilled. Then there exists an adjoint state $\bar{p} \in L^1([0,T],W^{1,s}(\Omega)) ,\; s\in (1,N/(N-1))$, and a Lagrange multiplier $\bar{\mu} \in \mathcal{M}(\overline{\Omega}_T)$ with $\bar{\mu}:=\mu_{\Omega_T}+\mu_{\Sigma_T}$, such that the following optimality system is fulfilled,\\
					\begin{equation}
		\begin{aligned}
	\bar{y}_t+A\bar{y}&=\bar{u},   \quad&\rm{in} \; \Omega_T;\\
	\partial_{\nu_A}\bar{y}&=\bar{v},  \quad&\rm{on} \; \Sigma_T;\\
	\bar{y}(\cdot,0)&=y_0,  \quad&\rm{on} \;\;\; \overline{\Omega};\\
	-\bar{p}_t+A\bar{p}&=\mu_{\Omega_T},  \; \quad&\rm{in} \; \Omega_T;\\
	\partial_{\nu_A}\bar{p}&=\mu_{\Sigma_T}, \; &\rm{on} \; \Sigma_T;\\
	\bar{p}(\cdot,T)&=\bar{y}(\cdot,T)-y_d, \; &\rm{on} \;\;\; \overline{\Omega};\\
	\int_{0}^{T}(\bar{p}+\alpha\bar{u},u-\bar{u})_{\Omega}dt&\geq0, \;\;\; &\forall u \in U_{ad};\\
	\int_{0}^{T}(\bar{p}+\alpha\bar{v},v-\bar{v})_{\Sigma}dt&\geq0, \;\;\;& \forall v \in V_{ad};\\
	\langle\bar{\mu},\bar{y}-\psi\rangle&=0, \;&\bar{\mu}\geqslant0.
		\end{aligned}
			\label{3}
				\end{equation}\\
		\end{theorem}
		\begin{proof}
	The proof of this Theorem is similar to Theorem 1.51 on page 80 in \cite{Hinze2009O} 
	\end{proof}
	The definition of the weak solution for the adjoint variable $\bar{p}$ is presented below.
	\begin{definition}
	A function $\bar{p} \in L^1([0,T],W^{1,s}(\Omega)) , \; s\in (1,N/(N-1))$ is called a weak solution of the state equation if it holds\\
	\[
	\begin{aligned}
		&\int_0^T\int_\Omega -\bar{p}_t(x,t)q(x,t)dxdt 
		- \int_0^T\int_\Sigma \mu_{\Sigma_T}q(x,t)dSdt \\
		&\quad + \int_0^T\int_\Omega \sum_{i,j=1}^{N} a_{ij}(x)\partial_{x_i}\bar{p}(x,t)\partial_{x_j}q(x,t)dxdt \\
		&= \int_0^T\int_\Omega \mu_{\Omega_T}q(x,t)dxdt,
		\quad \forall q \in L^\infty([0,T],W^{-1,s'}(\Omega)), \;\frac{1}{s'}+\frac{1}{s}=1 .
	\end{aligned}
	\]
	\end{definition}
		\begin{theorem}
Let $\bar{\mu} \in \mathcal{M}(\overline{\Omega}_T)$ be a regular Borel measure with $\bar{\mu}=\mu_{\Omega_T}+\mu_{\Sigma_T}$. There exists a unique solution  $p$ of the adjoint equation\\
	\[
\begin{aligned}
	-\bar{p}_t+A\bar{p}&=\mu_{\Omega_T},  \; &\rm{in} \; \Omega_T;\\
	\partial_{\nu_A}\bar{p}&=\mu_{\Sigma_T}, \; &\rm{on} \; \Sigma_T;\\
	\bar{p}(\cdot,T)&=\bar{y}(\cdot,T)-y_d, \;& \rm{on} \;\;\; \overline{\Omega}, \\
	    \end{aligned}
	\]
	where $\bar{p} \in L^r([0,T],W^{1,p}(\Omega)),\; r,p \in [1,2)$ with $(2/r)+(N/p) >N+1$. Then it holds that,\\
						\begin{equation}
		\begin{aligned}
	\parallel \bar{p}\parallel_{L^r([0,T],W^{1,p}(\Omega))} \leqslant C(\parallel \bar{\mu} \parallel_{\mathcal{M}(\bar{\Omega}_T)}+\parallel \bar{y}(\cdotp,T) \parallel_{L^2(\Omega)}+ \parallel y_d \parallel_{L^2(\Omega)}),
		\end{aligned}
					\end{equation}
	where constant $C$ is independent of $\bar{\mu}$, $\bar{y}(\cdotp,T)$ and $y_d$.
					\end{theorem}
	\begin{proof}
	The Theorem has been proved in Theorem 6.3 of \cite{Casas1997P}.
	\end{proof}
	\section{The augmented Lagrange method}
	Since the Lagrange multiplier corresponding to the inequality constraint $y \leq \psi$ in problem \eqref{P} belongs to the Radon measure space. It is difficult to deal with it numerically in the actual problem, so we add inequality constraint $y \leq \psi$ as a penalty term to the objective function to solve the problem.
	\subsection{The augmented Lagrange optimal control problem}
	Let $\rho>0$ be a given penalty parameter, and let $\mu \in L^2(\Omega_T)$ with $\mu>0$ be a given Lagrange multiplier. By reference \cite{B2019L}, we get the following augmented Lagrange function of \eqref{P}
		\[
	\begin{aligned}
	L_{\rho}(y,u,v,\mu):=&J(y,u,v)+\frac{\rho}{2}\parallel y-\psi+\frac{\mu}{\rho}-\mathcal{P}_{K}(y-\psi+\frac{\mu}{\rho})\parallel^2_{L^2(\Omega_T)}\\
	&-\frac{1}{2\rho}\parallel \mu\parallel^2_{L^2(\Omega_T)},
		    \end{aligned}
	\]
    for every $(y,u,v,\mu) \in \mathcal{Y} \times \mathcal{U} \times \mathcal{V} \times L^2(\Omega_T).$ Where $\mathcal{P}_K$ is the projection function from $L^2(\Omega_T)$ to $K=\{y \in L^2(\Omega_T)\mid y\leq 0\}$, which can be simply expressed as $\mathcal{P}_K(y)=\min\{0,y\}$. Therefore the augmented Lagrange function can be expressed as:\\
	$$L_{\rho}(y,u,v,\mu):=J(y,u,v)+\frac{1}{2\rho}\int_{\Omega_T}(\rho(y-\psi)+\mu)_+^2-\mu^2dxdt,$$
   Where $(y)_+:=\max(y,0)$. Therefore, when $\rho>0$ and $ \mu \in L^2(\Omega_T)$ are already given, The original problem \eqref{P} is transformed into the minimization of the augmented Lagrange function under the equations of state constraint in each step of the sub-problem.  The sub-problem of the augmented Lagrange method can be expressed in the following form:\\
\begin{equation}
	\begin{aligned}
		\min \;\;\;&L_{\rho}(y_\rho, u_\rho, v_\rho, \mu)\\
	  \rm{s.t.} \quad &y_\rho=S(u_\rho,v_\rho),\; u_\rho\in U_{ad},\; v_\rho\in V_{ad}.
	  	\end{aligned}
	  \tag{$P_{\rho,\mu}$}
	  \label{eq:P_rho_mu}
  \end{equation}
	We give the following theorem to guarantee the existence of optimal control $(\bar{u}_\rho,\bar{v}_\rho)$ in problem \eqref{eq:P_rho_mu}.
		\begin{theorem}
	For every $\rho >0$ and every $ \mu \in L^2(\Omega_T)$ with $\mu\geq 0$ the augmented Lagrange control problem \eqref{eq:P_rho_mu} admits a unique solution $(\bar{u}_\rho,\bar{v}_\rho)\in \mathcal{U}_{ad}\times \mathcal{V}_{ad}$ with associated optimal state $y_\rho \in \mathcal{Y}$.
		\end{theorem}
		\begin{proof}
	The theorem can be found in \cite{DelosReyes2015N,Tr2010O}.
	\end{proof}
    As the inequality state constraint in problem \( \eqref{P} \) is incorporated into the augmented Lagrange function, the Slater condition is not required for deriving the optimality condition. Furthermore, the first-order necessary condition for the optimal sub-problem \( \eqref{eq:P_rho_mu} \) is established through the following theorem.
			\begin{theorem}
 For every $\rho >0$ and every $ \mu \in L^2(\Omega_T)$ with $\mu\geq 0$, let $(\bar{y}_\rho,\bar{u}_\rho,\bar{v}_\rho)$ be the solution of \eqref{eq:P_rho_mu}. Then there exists a unique adjoint state $\bar{p}_\rho \in L^2(0,T;H^{-1}(\Omega))$, satisfying the following system,
				\begin{equation}
		\begin{aligned}
	(\bar{y}_\rho)_t+A\bar{y}_\rho&=\bar{u}_\rho,  \; &\rm{in} \; \Omega_T;\\
	\partial_{\nu_A}\bar{y}_\rho&=\bar{v}_\rho, \; &\rm{on} \; \Sigma_T;\\
	\bar{y}_\rho(\cdot,0)&=y_0, \; &\rm{on} \;\;\; \overline{\Omega};\\
	-(\bar{p}_\rho)_t+A\bar{p}_\rho&=\bar{\mu}_\rho,  \; &\rm{in} \; \Omega_T;\\
	\partial_{\nu_A}\bar{p}_\rho&=0, \; &\rm{on} \; \Sigma_T;\\
    \bar{p}_\rho(\cdot,T)&=\bar{y}_\rho(\cdot,T)-y_d, \; &\rm{on} \;\;\; \overline{\Omega};\\
	\int_{0}^{T}(\bar{p}_\rho+\alpha\bar{u}_{\rho},u-\bar{u}_{\rho})_{\Omega}dt&\geq0, \;\;\; &\forall u \in U_{ad};\\
    \int_{0}^{T}(\bar{p}_\rho+\alpha\bar{v}_{\rho},v-\bar{v}_{\rho})_{\Sigma}dt&\geq0, \;\;\; &\forall v \in V_{ad};\\
	\bar{\mu}_\rho&:=(\rho(\bar{y}_\rho-\psi)+\mu)_+.
		\end{aligned}
		\label{5}
			\end{equation}
					\end{theorem}
					\begin{proof}
		 We can find it in Corollary 1.3 on page 73 of \cite{Hinze2009O}.
			\end{proof}
	\subsection{The augmented Lagrange algorithm in brief}
    In this section, a brief introduction to the augmented Lagrange algorithm is provided. For convenience, the sub-problem \( (P_{\rho_k, \mu_k}) \) at iteration \( k \) is abbreviated as \( (P_k) \). Similarly, the solution \( (\bar{y}_{\rho_k}, \bar{u}_{\rho_k}, \bar{v}_{\rho_k}) \), the dual variable \( \bar{p}_{\rho_k} \), and the Lagrange multiplier \( \bar{\mu}_{\rho_k} \) corresponding to iteration \( k \) are denoted by \( (\bar{y}_k, \bar{u}_k, \bar{v}_k, \bar{p}_k, \bar{\mu}_k) \). The following Algorithm outlines a concise augmented Lagrange method.
\begin{algorithm}
	 \caption{} 
	 \begin{algorithmic}[1]
		 \State \textbf{Input:} $\rho_1 > 0$, $\mu_1 \in L^2(\Omega_T)$ with $\mu_1 \geq 0$. Choose $\gamma > 1$. 
		 \State Initialize $k = 1$. 
		 \Repeat 
		 \State Solve $(P_k)$, and obtain $(\bar{y}_k, \bar{u}_k, \bar{v}_k, \bar{p}_k)$. \State Set $\bar{\mu}_k := (\rho_k (\bar{y}_k - \psi) + \mu_k)_+$.
		  \If{the step is successful}
		   \State Set $\mu_{k+1} := \bar{\mu}_k$, $\rho_{k+1} := \rho_k$. 
		   \Else 
		   \State Set $\mu_{k+1} := \bar{\mu}_k$, increase penalty parameter $\rho_{k+1} := \gamma \rho_k$.
		    \EndIf 
		    \State Set $k := k + 1$. 
		    \Until{stopping criterion is satisfied.}
	     \end{algorithmic} 
     \end{algorithm} \\
     Two key issues remain to be addressed in this algorithm. The first is defining the successful step and specifying the termination condition, which will be discussed in the next section. The second is solving the sub-problem \( (P_k) \). The MSA algorithm can be utilized to obtain the optimal solution of \( (P_k) \) by minimizing the control of the Hamiltonian. A detailed analysis of this procedure will be provided in Section 4.
	\subsection{The multiplier update rule}
    First, we present a lemma that establishes how the control variables \( u \) and \( v \) can be governed by the multipliers and state constraints.
	\begin{lemma}
	Let $(\bar{y},\bar{u},\bar{v},\bar{\mu})$ be a solution of \eqref{3} and let $(\bar{y}_k,\bar{u}_k,\bar{v}_k,\bar{\mu}_k)$ solve \eqref{5} at k step. Then it holds\\
					\begin{equation}
		\begin{aligned}
	\alpha \parallel \bar{u}-\bar{u}_k\parallel^2_{L^2(\Omega_T)}+\beta \parallel \bar{v}-\bar{v}_k \parallel^2_{L^2(\Sigma_T)}\\
	\leq  \int^T_0(\bar{\mu}_{k},\psi-\bar{y}_k)_{\Omega}dt+\langle \bar{\mu},\bar{y}_k-\psi \rangle.
			\end{aligned}
					\end{equation}
						\end{lemma}
						\begin{proof}
	The following result can be derived from the definitions in \eqref{3} and \eqref{5}, together with the weak formulations of the state equation and the adjoint equation.
\begin{equation}
	\begin{aligned} 
		&\alpha \|\bar{u} - \bar{u}_k\|^2_{L^2(\Omega_T)} + \beta \|\bar{v} - \bar{v}_k\|^2_{L^2(\Sigma_T)} \\
		= &\int_0^T \big[ \alpha (\bar{u} - \bar{u}_k, \bar{u} - \bar{u}_k)_{\Omega} 
		+ \beta (\bar{v} - \bar{v}_k, \bar{v} - \bar{v}_k)_{\Sigma} \big] \, dt \\ 
		\leq &\int_0^T \big[ - (\bar{p} - \bar{p}_k, \bar{u} - \bar{u}_k)_{\Omega} 
		- (\bar{p} - \bar{p}_k, \bar{v} - \bar{v}_k)_{\Sigma} \big] \, dt \\
		= &\int_0^T \big[ - (\bar{p} - \bar{p}_k, \bar{y}_t + A\bar{y} - (\bar{y}_k)_t - A\bar{y}_k)_{\Omega} 
		- (\bar{p} - \bar{p}_k, \bar{v} - \bar{v}_k)_{\Sigma} \big] \, dt \\ 
		= &\int_0^T \big[ - (A(\bar{p} - \bar{p}_k), \bar{y} - \bar{y}_k)_{\Omega} 
		- (\bar{p} - \bar{p}_k, \bar{y}_t - (\bar{y}_k)_t)_{\Omega} 
		- (\mu_{\Sigma}, \bar{y} - \bar{y}_k)_{\Sigma} \big] \, dt \\ 
		= &\int_0^T \big[ - (\mu_{\Omega} - \bar{\mu}_k, \bar{y} - \bar{y}_k)_{\Omega} 
		- (\bar{p}_t - (\bar{p}_k)_t, \bar{y} - \bar{y}_k)_{\Omega} \\
		&\quad - (\bar{p} - \bar{p}_k, \bar{y}_t - (\bar{y}_k)_t)_{\Omega} 
		- (\mu_{\Sigma}, \bar{y} - \bar{y}_k)_{\Sigma} \big] \, dt \\ 
		= &\int_0^T - (\bar{\mu} - \bar{\mu}_k, \bar{y} - \bar{y}_k)_{\Omega} \, dt 
		- (\bar{p}(T) - \bar{p}_k(T), \bar{y}(T) - \bar{y}_k(T))_{\Omega} \\ 
		= &\int_0^T (\bar{\mu}_k - \bar{\mu}, \bar{y} - \bar{y}_k)_{\Omega} \, dt 
		- \|\bar{y}(T) - \bar{y}_k(T)\|^2_{L^2(\Omega)} \\ 
		\leq &\int_0^T (\bar{\mu}_k - \bar{\mu}, \bar{y} - \bar{y}_k)_{\Omega} \, dt. 
		\label{7}
	\end{aligned}
\end{equation}

\noindent
Here, the inequality in the second line is derived from inequalities \eqref{3} and \eqref{5}.

	The right side of the inequality can be split into two parts:\\
	\begin{equation}
	\begin{aligned}
	\int^T_0(\bar{\mu}_{k},\bar{y}-\bar{y}_k)_{\Omega}dt &=\int^T_0(\bar{\mu}_{k},\bar{y}-\psi)_{\Omega}dt+\int^T_0(\bar{\mu}_{k},\psi-\bar{y}_k)_{\Omega}dt \\
	&\leq \int^T_0(\bar{\mu}_{k},\psi-\bar{y}_k)_{\Omega}dt,
	\label{8}
\end{aligned}
\end{equation}
	and\\
		\begin{equation}
		\begin{aligned}
	-\langle \bar{\mu},\bar{y}-\bar{y}_k\rangle
	&=-\langle \bar{\mu},\bar{y}-\psi\rangle -\langle \bar{\mu},\psi-\bar{y}_k\rangle\\
	&=\langle \bar{\mu},\bar{y}_k-\psi\rangle,
	\label{9}
\end{aligned}
\end{equation}
	The following results can be obtained by combining \eqref{7}, \eqref{8} and \eqref{9} formulas\\
	$$\alpha \parallel \bar{u}-\bar{u}_k\parallel^2_{L^2(\Omega_T)}+\beta \parallel \bar{v}-\bar{v}_k \parallel^2_{L^2(\Sigma_T)}\leq  \int^T_0(\bar{\mu}_{k},\psi-\bar{y}_k)_{\Omega}dt+\langle \bar{\mu},\bar{y}_k-\psi\rangle.$$
\end{proof}
\begin{lemma}
Let $(\bar{y},\bar{u},\bar{v},\bar{\mu})$ and $(\bar{y}_k,\bar{u}_k,\bar{v}_k,\bar{\mu}_k)$ be given as the ones in Lemma 3.1. Then it holds\\
			\begin{equation}
		\begin{aligned}
	&\alpha \parallel \bar{u}-\bar{u}_k\parallel^2_{L^2(\Omega_T)}+\beta \parallel \bar{v}-\bar{v}_k \parallel^2_{L^2(\Sigma_T)}\\
	&\leq \parallel \bar{\mu}\parallel_{\mathcal{M}(\overline{\Omega}_T)}\parallel (\bar{y}_k-\psi)_+\parallel_{C(\overline{\Omega}_T)}+|\int^T_0(\bar{\mu}_{k},\psi-\bar{y}_k)_{\Omega}dt|.
\end{aligned}
\end{equation}
\end{lemma}
\begin{proof}
Apply the following estimates to lemma 3.4  \\
	$$\langle \bar{\mu},\bar{y}_k-\psi\rangle \leq \parallel \bar{\mu}\parallel_{M(\overline{\Omega}_T)}\parallel (\bar{y}_k-\psi)_+\parallel_{C(\overline{\Omega}_T)}.$$
\end{proof}
    From Lemma 2, an upper bound for the error between the control \( (\bar{u}_k, \bar{v}_k) \) at iteration \( k \) and the optimal control \( (\bar{u}, \bar{v}) \) is established in the \( L^2 \)-norm. Therefore when the quantity
	$$\parallel (\bar{y}_k-\psi)_+\parallel_{C(\overline{\Omega}_T)}+|\int^T_0(\bar{\mu}_{k},\psi-\bar{y}_k)_{\Omega}dt|$$
	tends to zero for $k\rightarrow0$, then $(\bar{u}_k,\bar{v}_k)$ converges strongly to $(\bar{u},\bar{v})$ at $L^2$-norm. When the condition \\
\[
\begin{aligned}
	&\parallel (\bar{y}_k-\psi)_+\parallel_{C(\overline{\Omega}_T)}+|\int^T_0(\bar{\mu}_{k},\psi-\bar{y}_k)_{\Omega}dt|\\
	&\leq \tau(\parallel (\bar{y}_{n}-\psi)_+\parallel_{C(\overline{\Omega}_T)}+|\int^T_0(\bar{\mu}_{n},\psi-\bar{y}_{n})_{\Omega}dt|
\end{aligned}
\]
      is satisfied with \( \tau \in (0,1) \), the penalty factor \( \rho_k \) selected at step \( k \) is considered appropriate. In this case, the penalty factor \( \rho_k \) remains unchanged, and the multiplier \( \mu_k \) is updated. Otherwise, \( \rho_k \) is updated while keeping the multiplier \( \mu_k \) fixed. If the above quantity becomes sufficiently small, the iteration process can be terminated.
	\subsection{The augmented Lagrange algorithm in detail}
     We now utilize the multiplier and penalty update rule, along with the termination condition derived in the previous section, to formulate the following algorithm.
\begin{algorithm}
	\caption{ }
	\begin{algorithmic}[1]
		\State \textbf{Input:} $\rho_1 > 0$, $\mu_1 \in L^2(\Omega_T)$ with $\mu_1 \geq 0$. Set $\gamma > 1$, $\tau \in (0,1)$, $R^+_0 \gg 1$, $\epsilon \geq 0$, $k = 1$, and $n = 1$.
		\Repeat
		\State \textbf{Step 1.} Solve $(P_k)$, and obtain $(\bar{y}_k, \bar{u}_k, \bar{v}_k, \bar{p}_k)$.
		\State \textbf{Step 2.} Set $\bar{\mu}_k := (\rho_k (\bar{y}_k - \psi) + \mu_k)_+$.
		\State \textbf{Step 3.} Compute $R_k := \|\bar{y}_k - \psi\|_{C(\overline{\Omega}_T)} + \left| \int_0^T (\bar{\mu}_{k}, \psi - \bar{y}_k) \, dt \right|$.
		\If{$R_k \leq \tau R^+_{n-1}$}
		\State \textbf{Step 4.} Set $\mu_{k+1} := \bar{\mu}_k$, $\rho_{k+1} := \rho_k$, and define $(y^+_n, u^+_n, v^+_n, p^+_n, \mu^+_n, R^+_n) := (\bar{y}_k, \bar{u}_k, \bar{v}_k, \bar{p}_k, \bar{\mu}_k, R_k)$. Set $n:= n + 1$.
		\Else
		\State \textbf{Step 5.} Set $\mu_{k+1} := \mu_k$, $\rho_{k+1} := \gamma \rho_k$.
		\EndIf
		\State \textbf{Step 6.} Update $k := k + 1$.
		\Until{$R^+_n \leq \epsilon$}.
	\end{algorithmic}
\end{algorithm}\\
     Before proving the convergence of the solution to Algorithm 2, we first present a lemma that demonstrates the iteration of Algorithm 2 can be terminated after a sufficient number of iterations. The following lemma is provided.
	\begin{lemma}
    Let \( R^+_n \) be generated by Algorithm 2. Then, as \( k \to \infty \), \( R^+_n \to 0 \), where \( n \) denotes the number of successful iterations and is related to \( k \).
		\end{lemma}
		\begin{proof}
To prove the above lemma, we consider two cases : $n \rightarrow \infty$ and $n$ stays bounded.\\
	Case 1 $(n \rightarrow \infty)$: By the definition of the algorithm:\\
	$$0 \leqslant R^+_{n} \leqslant \tau^nR^+_{0}, $$
	Passing to the limit $n \rightarrow \infty$ and $\tau \in (0,1)$ yields
	$$R^+_{n} \rightarrow 0.$$
	Case 2 ($n$ stays bounded): We assume that the maximum value of $n$ is $N$. From the algorithm we know exist a constant $M$, $R_k >\tau R^+_{N}$, when $k > M$. Set $\mu_M=\mu$, then when $k > M$, the solution $(\bar{y}_k,\bar{u}_k,\bar{v}_k,\bar{p}_k)$ of problem $(P_k)$ satisfies the following necessary conditions\\
			\begin{equation}
	\begin{aligned}
		(\bar{y}_k)_t+A\bar{y}_k&=\bar{u}_k,  \; &\rm{in} \; \Omega_T;\\
		\partial_{\nu_A}\bar{y}_k&=\bar{v}_k, \; &\rm{on} \; \Sigma_T;\\
		\bar{y}_k(\cdot,0)&=y_0, \; &\rm{in} \;\;\; \overline{\Omega};\\
		-(\bar{p}_k)_t+A\bar{p}_k&=\bar{\mu}_k,  \; &\rm{in} \; \Omega_T;\\
		\partial_{\nu_A}\bar{p}_k&=0, \; &\rm{on} \; \Sigma_T;\\
		\bar{p}_k(\cdot,T)&=\bar{y}_k(\cdot,T)-y_d, \; &\rm{in} \;\;\; \overline{\Omega};\\
		\int_{0}^{T}(\bar{p}_k+\alpha\bar{u}_{k},u-\bar{u}_{k})_{\Omega}dt&\geq0, \;\;\; &\forall u \in U_{ad};\\
		\int_{0}^{T}(\bar{p}_k+\alpha\bar{v}_{k},v-\bar{v}_{k})_{\Sigma}dt&\geq0, \;\;\; &\forall v \in V_{ad};\\
		\bar{\mu}_k&:=(\rho_k(\bar{y}_k-\psi)+\mu)_+.
	\end{aligned}
	\label{11}
\end{equation}\\
So first we have to prove $\lim\limits_{k\rightarrow \infty}\parallel(\bar{y}_k-\psi)_+\parallel_{C(\overline{\Omega}_T)}=0$. By imitating the proof of \cite{Ito2003S,Karl2018A}, we can get\\
			\begin{equation}
	\begin{aligned}
    \int^T_0(\bar{\mu}_k,\psi-\bar{y}_k)_{\Omega}dt&=-{\frac{1}{\rho_k}}  \int^T_0(\bar{\mu}_k,\rho_k(\bar{y}_k-\psi)+\mu-\mu)_{\Omega}dt\\
     &= -{\frac{1}{\rho_k}}\parallel \bar{\mu}_k\parallel^2_{L^2(\Omega_T)}+{\frac{1}{\rho_k}}  \int^T_0(\bar{\mu}_k,\mu)_{\Omega}dt \\
     &\leqslant -{\frac{1}{2\rho_k}}\parallel \bar{\mu}_k\parallel^2_{L^2(\Omega_T)}+{\frac{1}{2\rho_k}}\parallel \mu \parallel^2_{L^2(\Omega_T)}.
     	\end{aligned}
     \label{12}
 \end{equation}\\
 From \eqref{12}, Lemma 3.3 \eqref{7} and Theorem 2.2 \eqref{2}, we can get\\
 			\begin{equation}
 	\begin{aligned}
     &\alpha \parallel \bar{u}-\bar{u}_k\parallel^2_{L^2(\Omega_T)}+\beta \parallel \bar{v}-\bar{v}_k \parallel^2_{L^2(\Sigma_T)}\\
      &\leqslant \int^T_0(\bar{\mu}_k,\psi-\bar{y}_k)_{\Omega}dt - \langle \bar{\mu},\bar{y}-\bar{y}_k\rangle\\
     & \leqslant -{\frac{1}{2\rho_k}}\parallel \bar{\mu}_k\parallel^2_{L^2(\Omega_T)}+{\frac{1}{2\rho_k}}\parallel \mu \parallel^2_{L^2(\Omega_T)} +\parallel \bar{\mu}\parallel_{\mathcal{M}(\overline{\Omega}_T)}\parallel \bar{y}-\bar{y}_k\parallel_{C(\overline{\Omega}_T)}\\
     & \leqslant -{\frac{1}{2\rho_k}}\parallel \bar{\mu}_k\parallel^2_{L^2(\Omega_T)}+{\frac{1}{2\rho_k}}\parallel \mu \parallel^2_{L^2(\Omega_T)} \\
     &\;\;\;\;+ C_\infty\parallel \bar{\mu}\parallel_{\mathcal{M}(\overline{\Omega}_T)}( \parallel \bar{u}-\bar{u}_k\parallel_{L^r(\Omega_T)}+ \parallel \bar{v}-\bar{v}_k\parallel_{L^s(\Sigma_T)}).
	\end{aligned}
\label{13}
\end{equation}\\
Because of the boundedness of the control regions $\mathcal{U}_{ad}$ and $\mathcal{V}_{ad}$, we know in \eqref{13} that ${\frac{1}{2\rho_k}}\parallel \bar{\mu}_k\parallel^2_{L^2(\Omega_T)}$ is also bounded, where ${\frac{1}{2\rho_k}}\parallel \bar{\mu}_k\parallel^2_{L^2(\Omega_T)}$ can be expressed as\\
 			\begin{equation}
	\begin{aligned}
      {\frac{1}{2\rho_k}}\parallel \bar{\mu}_k\parallel^2_{L^2(\Omega_T)} =\frac{\rho_k}{2}\parallel (\frac{\mu}{\rho_k}+\bar{y}_k-\psi)_+\parallel^2_{L^2(\Omega_T)}.
    \end{aligned}
      \label{14}
      \end{equation}\\
Because $\rho_k \to \infty$, it follows that
\[
\limsup_{k \to \infty} \left( \frac{\mu}{\rho_k} + \bar{y}_k - \psi \right)_+ = 0,
\]
which implies
\[
\limsup_{k \to \infty} \bar{y}_k \leq \psi.
\]
This leads to the inequality
\begin{equation}
	\limsup_{k \to \infty} \langle \bar{y}_k - \psi, \bar{\mu} \rangle \leq 0.
	\label{15}
\end{equation}
Using \eqref{12} and \eqref{15}, along with \eqref{7} from Lemma 3.3, we obtain
\begin{equation}
	\begin{aligned}
		&\limsup_{k \to \infty} \big( \alpha \| \bar{u} - \bar{u}_k \|^2_{L^2(\Omega_T)} + \beta \| \bar{v} - \bar{v}_k \|^2_{L^2(\Sigma_T)} \big) \\
		&\leq \limsup_{k \to \infty} \Big[ \int_0^T (\bar{\mu}_k, \psi - \bar{y}_k)_{\Omega} \, dt 
		- \langle \bar{\mu}, \bar{y} - \psi \rangle 
		+ \langle \bar{y}_k - \psi, \bar{\mu} \rangle \Big] \\
		&\leq \limsup_{k \to \infty} \Big[ \frac{1}{2\rho_k} \| \mu \|^2_{L^2(\Omega_T)} 
		+ \langle \bar{y}_k - \psi, \bar{\mu} \rangle \Big] \\
		&\leq 0.
	\end{aligned}
	\label{50}
\end{equation}
Clearly,  
\[
\liminf\limits_{k \to \infty} \big( \alpha \| \bar{u} - \bar{u}_k \|^2_{L^2(\Omega_T)} + \beta \| \bar{v} - \bar{v}_k \|^2_{L^2(\Sigma_T)} \big) \geq 0.
\]  
This implies that \(\bar{u}_k \to \bar{u}\) in \(L^2(\Omega_T)\) and \(\bar{v}_k \to \bar{v}\) in \(L^2(\Sigma_T)\). Since $(\bar{u}_k,\bar{v}_k)$ converges strongly to $(\bar{u},\bar{v})$ in $L^2(\Omega_T)\times L^2(\Sigma_T)$, it is possible to extract a sub-sequence $(\bar{u}_{k'},\bar{v}_{k'})$, which converges to $\bar{u}$ almost everywhere. And due to the boundedness of $u_{k'}$ in $L^r(\Omega_T)$ and $u_{k'}$ in $L^s(\Sigma_T)$, we can find the control function $u_b-u_a$, $v_b-v_a$ such that\\
$$|\bar{u}_{k'}-\bar{u}|\leq u_b-u_a.$$
$$|\bar{v}_{k'}-\bar{u}|\leq v_b-v_a.$$
By the dominated convergence theorem, we get the sub-sequence $\bar{u}_{k'}$ converges strongly to $\bar{u}$ in $L^r(\Omega_T)$,  $\bar{v}_{k'}$ converges strongly to $\bar{v}$ in $L^s(\Sigma_T)$. As the limit is independent of the taken sub-sequence, we obtain convergence of the whole sequences $(\bar{u}_k,\bar{v}_k)$.\\
Moreover we can get $\bar{y}_k \to \bar{y}$ in $C(\overline{\Omega}_T) \cap W(0,T;L^2(\Omega),H^1(\Omega))$ by Theorem 2.2. Therefore\\
\begin{equation}
	\begin{aligned}
   \lim\limits_{k \to \infty}\parallel(\bar{y}_k-\psi)_+\parallel_{C(\overline{\Omega}_T)}=0
	\end{aligned}
	\label{51}
\end{equation}
      Next we will prove $\lim\limits_ {k\rightarrow \infty}{\int^T_0}(\bar{\mu}_k,\psi-\bar{y}_k) dt=0$. We can use \eqref{12} to estimate\\
      $$\int^T_0(\bar{\mu}_k,\psi-\bar{y}_k)_{\Omega}dt \leqslant -{\frac{1}{2\rho_k}}\parallel \bar{\mu}_k\parallel^2_{L^2(\Omega_T)}+{\frac{1}{2\rho_k}}\parallel \mu \parallel^2_{L^2(\Omega_T)} \leqslant {\frac{1}{2\rho_k}}\parallel \mu \parallel^2_{L^2(\Omega_T)}.$$
      which proves \\
       			\begin{equation}
      	\begin{aligned}
      \limsup_{k \rightarrow \infty}\int^T_0(\bar{\mu}_k,\psi-\bar{y}_k)_{\Omega}dt \leqslant 0.
      	\end{aligned}
      	\label{16}
      		\end{equation}\\
      		From Lemma 3.3 we get
      $$\int^T_0(\bar{\mu}_k,\psi-\bar{y}_k)dt \geqslant \alpha \parallel \bar{u}-\bar{u}_k\parallel^2_{L^2(\Omega_T)}+\beta \parallel \bar{v}-\bar{v}_k \parallel^2_{L^2(\Sigma_T)}+\langle \bar{\mu},\psi-\bar{y}_k\rangle,$$
      which leads with $\bar{\mu}\geq0$ and $\psi-\bar{y}_k\geq0$ to
        			\begin{equation}
      	\begin{aligned}
      \liminf_{k\rightarrow \infty}\int^T_0(\bar{\mu}_k,\psi-\bar{y}_k)_{\Omega}dt \geqslant 0.
            	\end{aligned}
      \label{17}
  \end{equation}\\
  The inequalities \eqref{16} and \eqref{17} yields
          			\begin{equation}
  	\begin{aligned}
      \lim_{k\rightarrow \infty}\int_{0}^{T}(\bar{\mu}_k,\psi-\bar{y}_k)_{\Omega}dt=0.
                  	\end{aligned}
      \label{18}
  \end{equation}\\
  Combining \eqref{51} and \eqref{18} we have\\
      $$\lim_{k\rightarrow \infty}R_k=\lim_{k\rightarrow \infty}(\parallel(\bar{y}_k-\psi)_+\parallel_{C(\overline{\Omega}_T)}+|\int^T_0(\bar{\mu}_{k},\psi-\bar{y}_k)_{\Omega}dt|)=0.$$
      This conclusion is inconsistent with $R_k<\tau R^+_N$, therefore Case 2 is not valid.
      \end{proof}
     \subsection{Convergence of algorithm}
     In this section, we first prove the convergence of $(y^+_n,u^+_n,v^+_n)$, given by the following theorem.
     \begin{theorem}
As $k\to \infty$ we have for the sequence $(y^+_n,u^+_n,v^+_n)$ generated by  Algorithm 2\\
\[
(y^+_n, u^+_n, v^+_n) \to (\bar{y}, \bar{u}, \bar{v}) 
\]
\[
\rm{in} \; \it \big(\mathcal{W}(0,T;L^2(\Omega),H^1(\Omega)) \cap C(\overline{\Omega}_T)\big) \times L^r(\Omega_T) \times L^s(\Sigma_T).
\]

	  \end{theorem}
	  \begin{proof}
	Due to Lemma 3.5, we know that when $k \rightarrow \infty$  get \\
	        			\begin{equation}
		\begin{aligned}
	\lim_{k\rightarrow \infty}R^+_n=\lim_{k\rightarrow \infty}(\parallel(y^+_n-\psi)_+\parallel_{C(\overline{\Omega}_T)}+|\int^T_0(\mu^+_{n},\psi-y^+_n)_{\Omega}dt|)=0.
	                  	\end{aligned}
	\label{19}
\end{equation}\\
	From Lemma 3.3, the following inequality holds\\
	\[
	\begin{aligned}
	&\alpha \parallel \bar{u}-u^+_n\parallel^2_{L^2(\Omega_T)}+\beta \parallel \bar{v}-v^+_n \parallel^2_{L^2(\Sigma_T)}\\
	&\leq \parallel \bar{\mu}\parallel_{M_{(\bar{\Omega}_T)}}\parallel (y^+_n-\psi)_+\parallel_{C_{(\bar{\Omega}_T)}}+|(\mu^+_n,\psi-y^+_n)|.
		\end{aligned}
		\]
	Using \eqref{19} from above we conclude\\
	$$0 \leq \lim_{k\rightarrow \infty} \alpha \parallel \bar{u}-u^+_n\parallel^2_{L^2(\Omega_T)}+\beta \parallel \bar{v}-v^+_n \parallel^2_{L^2(\Sigma_T)} \leq 0,$$
    Therefore, $(u^+_n, v^+_n) \to (\bar{u}, \bar{v})$ in $L^2(\Omega_T) \times L^2(\Sigma_T)$. Moreover, due to the boundedness of $(u^+_{n_k}, v^+_{n_k})$ in $L^r(\Omega_T) \times L^s(\Sigma_T)$, and similar to the proof in Lemma 3.3, we conclude that $(u^+_{n_k}, v^+_{n_k})$ converges strongly to $(\bar{u}, \bar{v})$ in $L^r(\Omega_T) \times L^s(\Sigma_T)$. \\
    Similarly, by Theorem 2.2, we obtain $y^+_n \to \bar{y}$ in $C(\overline{\Omega}_T) \cap \mathcal{W}(0,T;L^2(\Omega), H^1(\Omega))$.
		  \end{proof}
Having established the convergence of the internal control \( u^+_n \), the edge control \( v^+_n \), and the state \( y^+_n \) generated by Algorithm 2, we now proceed to analyze the convergence of the dual variables. Here, \( \mu^+_n \) denotes the approximation of the Lagrange multipliers, and \( p^+_n \) represents the approximation of the adjoint quantities.
To ensure the convergence of the dual variables, it is first necessary to establish the boundedness of \( \mu^+_n \) in \( L^1(\Omega_T) \) and the boundedness of \( p^+_n \) in \( L^1([0,T], W^{1,s}(\Omega)) \), where \( s \in (1, N/(N-1)) \). The required lemma to demonstrate these boundedness properties is provided below.
	\begin{lemma}
$y^+_n$ and $\mu^+_n$ are generated by algorithm 2. Then it holds\\
		        			\begin{equation}
		\begin{aligned}
	&|\int^T_0(\mu^+_n,\psi-y^+_n)_{\Omega}dt| \\
	&\leq \tau^{n-1}(\parallel(y^+_1-\psi)_+\parallel_{C(\overline{\Omega}_T)}+\parallel \mu^+_1\parallel_{L^2(\Omega_T)}\parallel y^+_1-\psi\parallel_{L^2(\Omega_T)}).
		                  	\end{aligned}
	\label{20}
\end{equation}
\end{lemma}
\begin{proof}
	By definition of Algorithm 2, we get \\
	$$R^+_n \leq \tau^{n-1}R^+_1,$$
	which can be written as\\
	\[
	\begin{aligned}
	&|\int^T_0(\mu^+_n,\psi-y^+_n)_{\Omega}dt| \leq \parallel(y^+_n-\psi)_+\parallel_{C(\overline{\Omega}_T)}+|\int^T_0(\mu^+_{n},\psi-y^+_n)_{\Omega}dt|\\
	&\leq  \tau^{n-1}(\parallel(y^+_1-\psi)_+\parallel_{C(\overline{\Omega}_T)}+|\int^T_0(\mu^+_{1},\psi-y^+_1)_{\Omega}dt|)\\
	&\leq \tau^{n-1}(\parallel(y^+_1-\psi)_+\parallel_{C(\overline{\Omega}_T)}+\parallel \mu^+_1\parallel_{L^2(\Omega_T)}\parallel y^+_1-\psi\parallel_{L^2(\Omega_T)}).
\end{aligned}
\]
\end{proof}
   We now proceed to demonstrate the boundedness of \( \mu^+_n \) and \( p^+_n \).
	\begin{lemma}
Let Assumption 2 be fulfilled. $p^+_n$ and $\mu^+_n$ are generated by algorithm 2, i.e.,there is a constant $C>0$ such that for all $n$ it holds \\
	$$\parallel p^+_n\parallel_{L^1([0,T],W^{1,s}(\Omega))} +\parallel \mu^+_n \parallel_{L^1(\Omega_T)} \leq C.$$
\end{lemma}
\begin{proof}
First let us prove that $\mu^+_n$ is bounded in $L^1(\Omega_T)$. Let $(\hat{y},\hat{u},\hat{v})$ fulfilled the Slater conditions given by Assumption 2, i.e., there exists $\sigma>0$, such that $\psi-\hat{y} >\sigma$. We can get the following estimate\\
		        			\begin{equation}
		\begin{aligned}
	&\sigma\parallel \mu^+_n\parallel_{L^1(\Omega_T)}  =\int^T_0\int_{\Omega}\mu^+_n\sigma dxdt\leq\int^T_0\int_{\Omega} \mu^+_n(\psi-\hat{y})dxdt\\  &\leq\int^T_0\int_{\Omega}\mu^+_n(\psi-y^+_n)dxdt+\int^T_0\int_{\Omega}\mu^+_n (y^+_n-\hat{y}) dxdt.
			                  	\end{aligned}
	\label{21}
\end{equation}\\
The first part of \eqref{21} can be estimated with Lemma 3.7 yielding\\
		        			\begin{equation}
	\begin{aligned}
	&\int^T_0\int_{\Omega}\mu^+_n(\psi-y^+_n)dxdt \leq |\int^T_0(\mu^+_n,\psi-y^+_n)_{\Omega}dt|\\
	&\leq \tau^{n-1}(\parallel(y^+_1-\psi)_+\parallel_{C([0,T],\overline{\Omega})}+\parallel \mu^+_1\parallel_{L^2(\Omega_T)}\parallel (y^+_n-\psi)_+\parallel_{L^2(\Omega_T)})\\
	&=:C.\\
				                  	\end{aligned}
	\label{22}
\end{equation}\\
The second part of \eqref{21} can be estimated as following\\
		        			\begin{equation}
	\begin{aligned}
	&\int^T_0\int_{\Omega}\mu^+_n (y^+_n-\hat{y}) dxdt=\int^T_0\int_{\Omega}(-(p^+_n)_t+Ap^+_n) (y^+_n-\hat{y}) dxdt\\
	&= \int^T_0 (-(p^+_n)_t,y^+_n-\hat{y})_{\Omega}dt+\int^T_0 (p^+_n,A(y^+_n-\hat{y}))_{\Omega}dt\\
	&=\int^T_0 (-(p^+_n)_t,y^+_n-\hat{y})_{\Omega}dt+\int^T_0 (p^+_n,u^+_n-\hat{u})_{\Omega}dt\\
	&+\int^T_0 (p^+_n,-((y^+_n)_t-(\hat{y}))_t)_{\Omega}dt\\
	&=\int^T_0 (p^+_n,u^+_n-\hat{u})_{\Omega}dt -(p^+_n(T),y^+_n(T)-\hat{y}(T))_{\Omega}\\
	&\leq -\alpha \int^T_0 (u^+_n,u^+_n-\hat{u})_{\Omega}dt -(y^+_n(T)-y_d,y^+_n(T)-\hat{y}(T))_{\Omega} \\
	&\leq -\frac{\alpha}{2}\parallel u^+_n\parallel_{L^2(\Omega_T)} +\frac{\alpha}{2}\parallel \hat{u}\parallel_{L^2(\Omega_T)}\\
	&+\frac{1}{2}\parallel \hat{y}(T)-y^+_n(T)\parallel_{L^2(\Omega)}+\frac{1}{2}\parallel y^+_n(T)-y_d\parallel_{L^2(\Omega)}.
		\end{aligned}
	\label{23}
\end{equation}\\
Combine \eqref{21}, \eqref{21} and \eqref{23} together yield\\
\[
	\begin{aligned}
	&\parallel \mu^+_n\parallel_{L^1(\Omega_T)}+\frac{\alpha}{2\sigma}\parallel u^+_n\parallel_{L^2(\Omega_T)} \\
	&\leq \frac{C}{\sigma} +\frac{\alpha}{2\sigma}\parallel \hat{u}\parallel_{L^2(\Omega_T)}+\frac{1}{2\sigma}\parallel \hat{y}(T)-y^+_n(T)\parallel_{L^2(\Omega)}+\frac{1}{2\sigma}\parallel y^+_n(T)-y_d\parallel_{L^2(\Omega)}
\end{aligned}
\]
Since the right-hand side of the inequality is bounded, it is possible to obtain the boundedness of $\mu^+_n$ under $L_1$ norm. And we have the embedding $L^1(\Omega_T)\hookrightarrow M(\Omega_T)$, combined with Theorem 2.5, the boundedness of $\parallel p^+_n\parallel_{L^1([0,T],W^{1,s}(\Omega))}$ holds.
\end{proof}
\begin{theorem}
Let sub-sequences $(p^+_{n_j},\mu^+_{n_j})$ of $(p^+_n,\mu^+_n)$ be given such that $\mu^+_{n_j} \rightharpoonup^* \tilde{\mu}$ in $\mathcal{M}(\bar{\Omega}_T)$ and $p^+_{n_j} \rightharpoonup \tilde{p}$ in $L^1([0,T],W^{1,s}(\Omega)),\; s \in (1,N/N-1)$. Then $(\tilde{\mu},\tilde{p})$ satisfies necessary condition \eqref{3} of the original problem \eqref{P}. If $(\bar{p},\bar{\mu})$ are uniquely determined Lagrange multipliers. Then $\mu^+_{n} \rightharpoonup^* \bar{\mu}$ in $\mathcal{M}(\bar{\Omega}_T)$ and $p^+_{n} \rightharpoonup^* \bar{p}$ in $L^1([0,T],W^{1,s}(\Omega))$.
\end{theorem}
\begin{proof}
This proof follows the methodology outlined in \cite{Karl2018A,Hinze2010V}. By Lemma 3.8, the sequences $\mu^+_{n}$ and $p^+_{n}$ are bounded. Consequently, we can extract sub-sequences such that  
\[
\mu^+_{n_j} \rightharpoonup^* \tilde{\mu} \quad \text{and} \quad p^+_{n_j} \rightharpoonup^* \tilde{p},
\]  
where the weak-* convergence of $\mu^+_{n_j}$ and  $p^+_{n_j}$ are understood in their respective function spaces. Our objective is to demonstrate that the limits $\tilde{\mu}$ and $\tilde{p}$ satisfy the conditions specified in \eqref{3}.  \\
We begin by showing that $\tilde{\mu}$ satisfies the constraints in \eqref{3}. Specifically, it is necessary to verify that  
\[
\langle \tilde{\mu}, \phi \rangle \geq 0 \quad \forall \phi \in C(\overline{\Omega}_T), \, \phi \geq 0,
\]  
and  
\[
\langle \tilde{\mu}, \bar{y} - \psi \rangle = 0.
\]  

From the Lagrange multiplier update rule in Algorithm 2, we know that $\mu^+_{n_j} \geq 0$, which implies\\
    	\begin{equation}
    	\begin{aligned}
    \int_{\Omega_T}\mu^+_{n_j}\phi dxdt\rightarrow\langle\tilde{\mu},\phi\rangle\geq0 \;,\forall\phi \in C(\overline{\Omega}_T), \quad \phi \geq0.
    	\end{aligned}
    \label{24}
\end{equation}\\
    And then because of $y^+_{n_j}\rightarrow\bar{y}$ in $C(\overline{\Omega}_T)$ and lemma 3.7, we obtain\\
    $$\lim\limits_{j\rightarrow\infty}|\int^T_0(\mu^+_{n_j},\psi-y^+_{n_j})dt|=|\langle\tilde{\mu},\bar{y}-\psi\rangle|=0.$$
Next, we proceed to prove that $\tilde{p}$ satisfies \eqref{3}. Since the embedding  
\[
C(\overline{\Omega}_T) \hookrightarrow L^\infty([0,T]; W^{-1,s'}), \quad \text{where } \frac{1}{s} + \frac{1}{s'} = 1,
\]  
is continuous, and \(L^\infty([0,T]; W^{-1,s'})\) is the dual space of \(L^1([0,T], W^{1,s}(\Omega))\), we will continue to use \(\phi\) as the test function. Additionally, leveraging \eqref{24}, we obtain
    \begin{align*}
    &\int_{\Omega_T}\mu^+_{n_j}\phi dxdt=\int_{\Omega_T}[-(p^+_{n_j})_t+Ap^+_{n_j}]\phi dxdt\\
   & \rightarrow\int_{\Omega_T}[-\tilde{p}_t+A\tilde{p}]\phi dxdt=\langle\tilde{\mu},\phi\rangle.
\end{align*}
Since $\tilde{\mu}$ satisfies \eqref{3}, $\tilde{p}$ also satisfies \eqref{3}.
\end{proof}
    \section{Optimization of sub-problems}
     In the first step of each iteration of Algorithm 2, the optimal solution of the \( k \)-th iteration sub-problem needs to be determined. This will be achieved using the Method of Successive Approximations (MSA), which is based on Pontryagin's Principle. Below, Pontryagin's Principle is presented under the constraints of the parabolic equations \cite{Raymond1999H}.
    \subsection{Pontryagin's Maximum Principle}
    First, the expression for the Hamiltonian of the \( k \)-th iteration subproblem is provided, including the distributed Hamiltonian function and the boundary Hamiltonian function, as follows:\\
    $$H_{\Omega_T}(y_k,u_k,p_k,\mu_k,\rho_k):=\frac{\alpha}{2}u_k^2+\frac{1}{2\rho_k}((\rho_k(y_k-\psi)+\mu_k)_+^2-\mu_k^2)+p_ku_k,$$
    for every $(x,t,y_k,u_k,p_k)\in \Omega \times (0,T)\times R\times R\times R,$\\
    $$H_{\Sigma_T}(y_k,v_k,p_k):=\frac{\beta}{2}v_k^2+p_kv_k,$$
    for every $(x,t,v_k,p_k)\in \Sigma \times (0,T)\times R\times R$.\\
    The following theorem establishes Pontryagin's Principle for optimal control problems subject to constraints imposed by parabolic equations.
    \begin{theorem}
Let $(\bar{y}_k,\bar{u}_k,\bar{v}_k)$ be the solution of $(P_k)$. Then there exists a unique adjoint state $\bar{p}_k \in L^2(0,T;H^{-1}(\Omega))$, satisfying the following equation.\\
\begin{align*}
	-(\bar{p}_\rho)_t + A\bar{p}_\rho &= \bar{\mu}_\rho,  & \rm{in} \; \Omega_T,\\
	\partial_{\nu_A}\bar{p}_\rho &= 0, & \rm{on}\; \Sigma_T,\\
	\bar{p}_\rho(\cdot,T) &= \bar{y}_\rho(\cdot,T) - y_d, & \rm{in} \;\;\;\overline{\Omega}.
\end{align*}
and such that\\
\begin{align*}
	&H_{\Omega_T}(\bar{y}_k,\bar{u}_k,\bar{p}_k,\mu_k,\rho_k)=\lim\limits_{u\in U_{ad}}H_{\Omega_T}(\bar{y}_k,u,\bar{p}_k,\mu_k,\rho_k),\\
	&\rm{for \; a.e.}  \; (x,t)\in \Omega_T,\\
		&H_{\Sigma_T}(\bar{y}_k,\bar{v}_k,\bar{p}_k,\mu_k,\rho_k)=\lim\limits_{v\in V_{ad}}H_{\Omega_T}(\bar{y}_k,v,\bar{p}_k,\mu_k,\rho_k),\\
	&\rm{for \; a.e.} \; (x,t)\in \Sigma_T.
\end{align*}
    \end{theorem}
    \begin{proof}
    	See \cite{Raymond1999H} for the proof.
    \end{proof}
    \subsection{Method of Successive Approximations}
With the theoretical foundation provided by Pontryagin's Principle, the Method of Successive Approximations (MSA), originally proposed by Chernousko and Lyubushin in 1982 \cite{Chernous1982M}, can be introduced. This iterative method alternates between propagation and optimization steps. The following algorithm applies MSA to solve the optimal sub-problem in the \( k \)-th iteration.
\\
\\
\\
\begin{algorithm}
	\caption{ }
	\begin{algorithmic}[1]
    	 	\State Initialize $\rho_k > 0$, $\mu_k \in L^2(\Omega_T)$ with $\mu_k \geq 0$, and $(u_{k,1}, v_{k,1}) \in L^r(\Omega_T) \times L^s(\Sigma_T)$, set $i = 1$. 
    	 	 \Repeat 
    	 	 \State Solve $(y_{k,i})_t + Ay_{k,i} = u_{k,i}$, $\partial_{\nu_A}y_{k,i} = v_{k,i}$, $y_{k,i}(\cdot,0) = y_0$,
    	 	 \State Set $\mu_{k,i} := (\rho(y_{k,i} - \psi) + \mu_{k})_+$, 
    	 	 \State Solve $-(p_{k,i})_t + Ap_{k,i} = \mu_{k,i}$, $\partial_{\nu_A}p_{k,i} = 0$, $p_{k,i}(\cdot,T) = y_{k,i}(\cdot,T) - y_d$,
\State Set 
\[
\begin{aligned}
	u_{k,i+1} &= \arg\min_{u \in U_{ad}} H_{\Omega_T}(x,t,y_{k,i},u,p_{k,i}), \\
	v_{k,i+1} &= \arg\min_{v \in V_{ad}} H_{\Sigma_T}(x,t,v,p_{k,i})
\end{aligned}
\]
         for each $(x,t) \in \overline{\Omega}_T$.
    	 	  \State $i := i + 1$. 
    	 	  \Until{termination criterion is met.} 
     	  \end{algorithmic} 
       \end{algorithm}\\
   To minimize the Hamiltonian, gradient descent is utilized to perform updates. The gradients of the distributed Hamiltonian function and the boundary Hamiltonian function with respect to the internal control \( u \) and the boundary control \( v \) are expressed as follows:\\
   $$\nabla_uH_{\Omega_T}=\alpha u_k +p_k+p_k(\rho_k(y_k-\psi)+\mu_k)_+,$$
   $$\nabla_vH_{\Sigma_T}=\beta v_k+p_k.$$  
   
     \section{Numerical tests}
     In this section, we present numerical results for a pointwise state-constrained optimal control problem of a parabolic equation, where $\Omega$ is a two-dimensional domain. The original problem \eqref{P} is solved using the augmented Lagrange method described in Algorithm 2, while the sub-problem $(P_k)$ is addressed via the MSA outlined in Algorithm 3. The MSA was terminated when\\
     $$|u_{k,i+1}-u_{k,i}|\leqslant \epsilon_1$$
     was satisfied, i.e., the gap in the control variable is small enough.\\
     The augmented Lagrange algorithm was stopped when\\
     $$R^+_n=\lim_{k\rightarrow \infty}(\parallel(y^+_n-\psi)_+\parallel_{C(\overline{\Omega}_T)}+|\int^T_0(\mu^+_{n},\psi-y^+_n)dt|)\leq \epsilon_2.$$
     We consider an optimal control problem with $\Omega_T=[0,1]\times[0,1]\times[0,1]$ given by\\
     \begin{align*}
     	\min J(y,u): &=\frac{1}{2}\parallel y(\cdot,T)-y_d\parallel_{L^2(\Omega)}^2+\frac{\alpha}{2}\parallel u\parallel_{L^2(\Omega_T)}^2
     	    \end{align*}
     
    \[
\text{s.t.}
    \begin{cases}
    	y_t - \Delta y = u & \text{in } \Omega_T, \\
    	\partial_{\nu_A} y = 0 & \text{on } \Sigma_T, \\
    	y(\cdot, 0) = y_0 & \text{in } \overline{\Omega}, \\
    	y(x,t) \leq \psi(x,t) & \text{in } \Omega_T, \\
    	u_a(x,t) \leq u(x,t) \leq u_b(x,t) & \text{in } \Omega_T.
    \end{cases}
    \]
    We choose $y_0=sin(\pi x)sin(\pi y)$, $y_d=\exp(-2\alpha\pi T)sin(\pi x)sin(\pi y)$, $\psi=1$ , $\alpha=1 $ , $u_a=-1$ , $u_b=1$ , $\tau=0.9$ , $\gamma=2$. The problem has A theoretical analytic solution $\bar{y}=\exp(-2\alpha\pi t)sin(\pi x)sin(\pi y)$. We make the initialization control $u_0=0$, to make the descent of $R^+_n$ more intuitive, we set $\mu_0=10$. In solving the parabolic equation of state and the dual parabolic equation, we use the finite difference method, we set the step size of time and space in each direction to $0.25$. We adjust the learning rate of the sub-problem gradient descent to the dynamic learning rate, which is initially $0.001$ and multiplied by $0.9$ every 100 iterations. At the same time we set $\epsilon_1$ and $\epsilon_2$ of the termination condition to $10^{-4}$. The results are shown below.\\
    \begin{figure}[ht]
    	\centering
    	\includegraphics[width=0.6\textwidth]{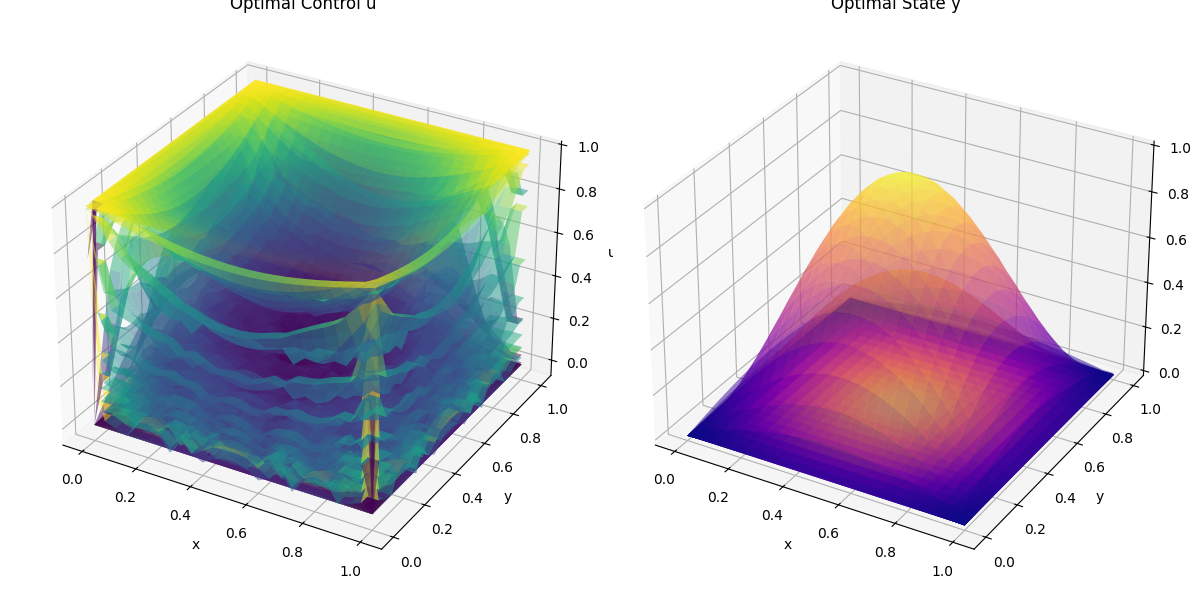} 
    	\caption{Computed discrete optimal state y (right) and optimal control u (left)} 
    	\label{fg1} 
    \end{figure}
        \begin{figure}[ht]
    	\centering
    	\includegraphics[width=0.6\textwidth]{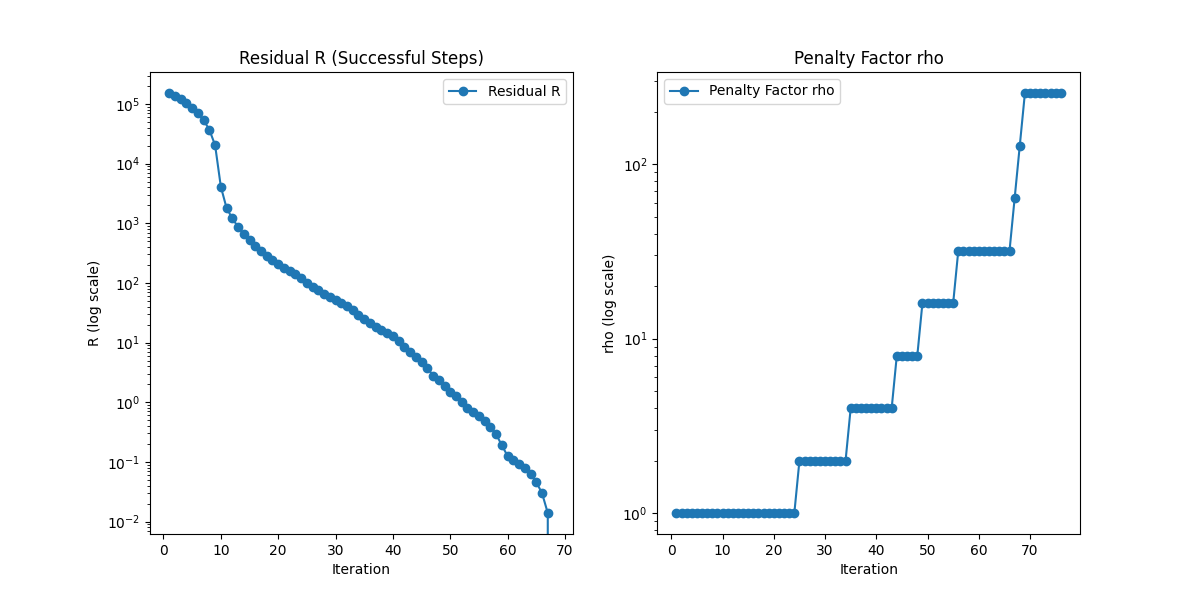} 
    	\caption{The number of successes n, the corresponding index $R^+_n$ (left), and the growth trend of penalty factor $\rho_k$ (right)} 
    	\label{fg2}
    \end{figure}
            \begin{figure}[ht]
    	\centering
    	\includegraphics[width=1\textwidth]{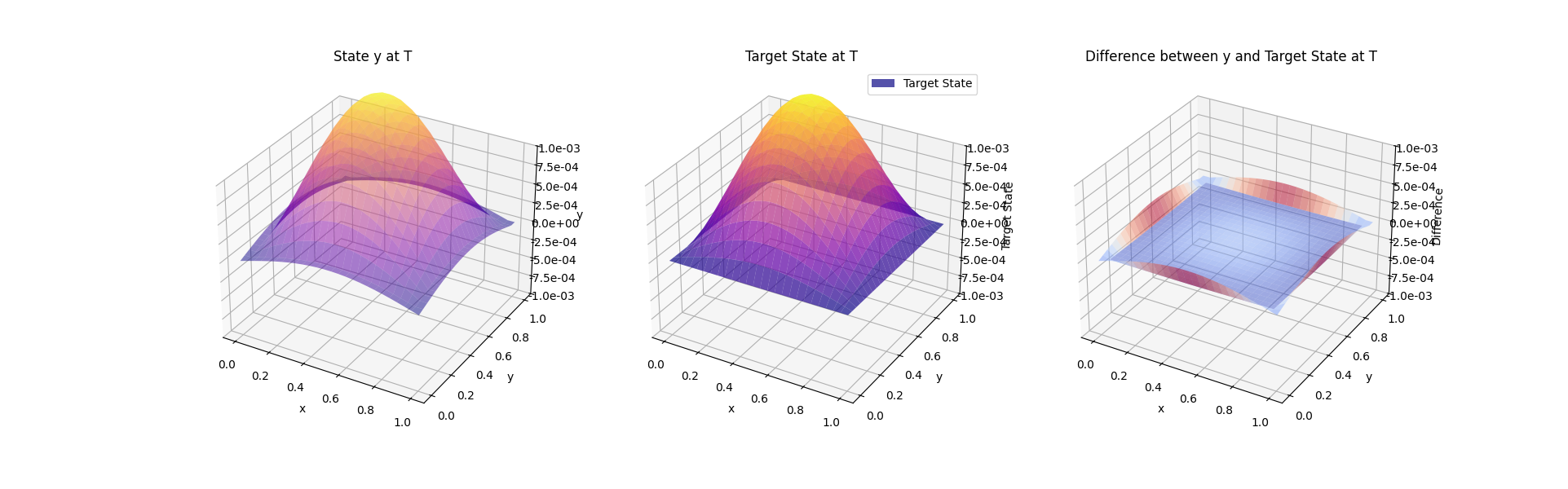} 
    	\caption{Computed discrete optimal state at time $T$ (left), Theoretical optimal state at time $T$ (center), Difference between computed and theoretical optimal state at time $T$ (right)} 
    	\label{fg3}
    \end{figure}
    \\
    From \ref{fg2}, we observe that $R^+_n$ approaches 0, demonstrating that the convergence performance of this algorithm is guaranteed. Additionally, from \ref{fg3}, we can see that the optimal state obtained by the algorithm is very close to the theoretical value of the optimal state, which further highlights the effectiveness of the algorithm.

     \section{Conclusion}
     In this paper, we apply the augmented Lagrange method in conjunction with the Method of Successive Approximations (MSA) to solve optimal control problems with state constraints and parabolic equation constraints. We introduce a convergence index to rigorously establish the convergence of the augmented Lagrange algorithm. The method presented in this work offers an effective alternative to the widely used regularization techniques, providing a more reliable and theoretically sound approach for solving such constrained optimal control problems.
\bibliographystyle{plain}
\bibliography{PALMH}

\end{document}